\documentclass[12pt,leqno]{amsart}
\usepackage{amsmath}
\usepackage{amsthm}
\allowdisplaybreaks
\usepackage{amssymb,color}

\def\underset#1#2{{\mathrel{\mathop {{}_{} {#2}}\limits_{{#1}_{}}}}}
\def\upplim_#1{\underset{#1}{\overline\lim}\;}
\def\lowlim_#1{\underset{#1}{\underline\lim}\;}
\newtheorem{corollary}[equation]{Corollary}

\newtheorem{lemma}[equation]{Lemma}

\newtheorem{proposition}[equation]{Proposition}

\newtheorem{theorem}[equation]{Theorem}

\newcommand{\C}{{\mathbb{C}}}

\renewcommand{\P}{{\mathbb{P}}}

\newcommand{\Z}{\mathbb{Z}}

\numberwithin{equation}{section}

\title[Meromorphic functions bi-weighted weakly sharing pairs of small functions]{Meromorphic functions bi-weighted weakly sharing pairs of small functions} 
\author{\sc Si Duc Quang$^{1.2}$}
\author{\sc Phung Nguyen Ngoc Anh$^1$}

\address{$^1$Department of Mathematics, School of Mathematics and Computer Science, Hanoi National University of Education, 136-Xuan Thuy - Cau Giay - Hanoi, Vietnam}
\address{$^2$Institute of Natural Sciences, Hanoi National University of Education, 136-Xuan Thuy - Cau Giay - Hanoi, Vietnam}
\email{quangsd@hnue.edu.vn;anhphungnguyenngoc@gmail.com}

\begin{document}

\begin{abstract} Two meromorphic functions $f$ and $g$ are said to weakly share a small function $a$ with bi-weight $(n,k)$ if the functions $f-a$ and $g-a$ have the same zeros with multiplicities truncated at level $n+1$, while zeros whose multiplicities exceed $k$ are disregarded. In this article, we show that if $f$ and $g$ weakly share three distinct small functions with suitable bi-weights and are not related by a quasi-M\"obius transformation, then for every other small function $c$, the counting function $N(r,\nu_f^c)$ is asymptotically equivalent to the characteristic function $T(r,f)$. Moreover, the truncated counting function $N_{(3}(r,\nu_f^c)$, which counts only zeros of multiplicity at least $3$, is negligible. As an application, we further prove that $f$ and $g$ must be related by a quasi-M\"obius transformation provided that they satisfy an additional condition, which is weaker than the usual assumption that they share a fourth pair of small functions.
\end{abstract}
\maketitle

\def\thefootnote{\empty}
\footnotetext{
Corresponding author: Si Duc Quang
}

\maketitle

\section{Introduction}

In 1926, Nevanlinna proved that two nonconstant meromorphic functions on $\C$ must be M\"obius transformations of each other if they share four distinct fixed values CM, that is, if they have the same inverse images of these values counting multiplicities. Since then, Nevanlinna's theorem has been substantially generalized in several directions, including the replacement of fixed values by small functions or pairs of small functions, as well as the weakening of the CM sharing condition. Many remarkable results in this area have been obtained through the contributions of numerous authors, including Czubiak-Gundersen \cite{CG}, Li-Yang \cite{LY1,LY2}, Quang-Quynh \cite{Q12,QQ}, Cao-Cao \cite{CC}, Zhang-Yan \cite{ZY}, and others. 

We now recall a theorem of Li and Yang from \cite{LY2}, which is among the most elegant results currently available in this direction. First, we introduce the following notation and definitions.

In this article, all meromorphic functions are defined on $\mathbb{C}$. A divisor $\nu$ on $\C$ is a function from $\C$ into $\Z$ so that the set $\overline{\{z|\nu(z)\ne 0\}}$ is a discrete set. For each positive integer $k$, we define
$$ \nu_{\le n}(z)=\begin{cases}
\nu (z)&\text{ if }\nu(z)< n,\\
n&\text{ if }\nu(z)\ge n. 
\end{cases} $$
We say that a meromorphic function $a$ is small with respect to a meromorphic function $f$ if their characteristic function satisfy $\|\ T(r,a)=o(T(r,f))$ (see Section 2 for the definition), where the notation ``$\|\ P$'' means the claim $P$ holds for all $r\in (0,+\infty)$ outside of a Borel subset of finite measure.  For a subset $S$ of $\C$, we define 
$$E^{S}_{n}(a,f)=\left\{\bigl (z,\nu^a_{f,\le n+1}(z)\bigl)\bigl |z\not\in S\right\}.$$
Here, by $\nu^a_{f}$ we denote the zero divisor of the function $f-a$. We will omit the character $^{S}$ (resp. $n$) in $E^{S}_{n}(a,f)$ if $S=\emptyset$ (resp. $n=+\infty$).

We say that two meromorphic functions $f$ and $g$ weakly share a pair of meromorphic functions $(a,b)$ with the weight $n$ if 
$$E^{S}_n(a,f)=E^{S}_n(b,g),$$
for a discrete subset $S$ of $\C$ with $\| N(r,\nu_S)=o(T(r,f)+T(r,g))$, where the divisor $\nu_S$ coincides with the characteristic function of the subset $S$ on $\C$ (see Section 2 for the definition of the counting function of a divisor $N(r,\nu)$). If $n=+\infty$ (resp. $n=1$) then we say that $f$ and $g$ weakly share the pair $(a,b)$ counting multiplicity (resp. ignoring multiplicity).

%Especially, if $S=\emptyset$ then we say that $f$ and $g$ share the pair $(a,b)$ counting multiplicity (resp. ignoring multiplicity).

The function $f$ is said to be a quasi-M\"{o}bius transformation of $g$ if there exist four small functions $\alpha_1, \alpha_2, \alpha_3, \alpha_4$ of $g$ with $\alpha_1\alpha_4-\alpha_2\alpha_3\not\equiv 0$ such that 
$$f=\dfrac{\alpha_1g+\alpha_2}{\alpha_3g+\alpha_4}.$$ 
If all $\alpha_i\ (1\le i\le 4)$ are constant then we say that $f$ is a M\"{o}bius transformation of $g$. The result of Li and Yang is stated as follows.

\vskip0.2cm
\noindent
\textbf{Theorem A} (See \cite[Theorem 1]{LY2}).\ {\it Let $f$ and $g$ be non-constant meromorphic functions and $a_i,b_i\ (i=1,2,3,4)\ (a_i\ne a_j, b_i\ne b_j, i\ne j)$ be small functions (with respect to $f$ and $g$). If $f$ and $g$ weakly share three pairs $(a_i, b_i),\ (i=1,2,3)$ counting multiplicity, and weakly share the fourth pair $(a_4, b_4)$ ignoring multiplicity, then $f$ is a quasi-M\"{o}bius transformation of $g$.}

Here, we note that in order to prove the above theorem, Li and Yang have used the following key lemma.

\vskip0.2cm
\noindent
\textbf{Lemma B} (Cf. \cite[Lemma 1]{LY2}).\ {\it Let $f$ and $g$ be two non-constant meromorphic functions, $a_1, a_2$ and $a_3$ be three distinct small functions with respect to $f$ and $g$. If $f$ and $g$ weakly share $a_1, a_2, a_3$ counting multiplicity, and if $f$ is not a quasi-M\"{o}bius transformation of $g$, then for any small function $c\ (\not\equiv a_1,a_2,a_3)$ with respect to $f$ and $g$, we have
$$\|\ T(r, f)=N(r, \nu^c_f)+o(T(r,f)+T(r,g)) \quad \text { and } \quad \|\ N_{(3}(r,\nu^c_f)=o(T(r,f)+T(r,g)).$$}
\indent
This lemma plays an important role in the study of two meromorphic functions sharing four pairs of small functions, and it is also of independent interest. Unfortunately, in \cite{LY2}, the proof of the lemma is omitted and the result is quoted from \cite[Lemma 5]{LY1} and \cite[Lemma 8]{LZ}.  

Although we believe that the lemma is correct, some parts of the arguments in \cite[Lemma 5]{LY1} and \cite[Lemma 8]{LZ} seem to require additional justification in the setting of weakly shared small functions. More precisely, in \cite{LZ}, the authors define
$$H_1=\frac{g-a_1}{f-a_1} \frac{f-a_3}{g-a_3}, \quad 
H_2=\frac{g-a_2}{f-a_2} \frac{f-a_3}{g-a_3}$$
and
$$H=\left(a_1-c\right)\left(a_2-a_3\right)\left(H_1-1\right)-\left(a_2-c\right)\left(a_1-a_3\right)\left(H_2-1\right).$$
Then
$$f-c=\frac{H}{\left(a_2-a_3\right)\left(H_1-1\right)-\left(a_1-a_3\right)\left(H_2-1\right)}.$$
Hence, if $f$ and $g$ weakly share $a_i\ (1\le i\le 3)$ counting multiplicity, then
$$\|\ \overline N(r,\nu^0_{H_i})+\overline N(r,\nu^\infty_{H_i})
=o(T(r,f)+T(r,g)).$$
However, it is not immediate that this further implies
$$\|\ N(r,\nu^0_{H_i})+N(r,\nu^\infty_{H_i})
=o(T(r,f)+T(r,g)),$$
as used in the proof of \cite[Lemma 8]{LZ}. Consequently, the last inequality
$$N_{(3}(r,\nu^c_f)\le N_{(3}(r,\nu^0_H)$$
may require additional explanation.  

Similarly, the proof of \cite[Lemma 5]{LY1} is clear in the case where $f$ and $g$ share $a_1,a_2,a_3$ CM in the usual sense. However, in the case where $f$ and $g$ only weakly share these small functions CM, additional care seems to be needed. More precisely, it is not completely clear that
$$\|\ N(r,\nu^c_f)
=N(r,\nu^0_H)
-N(r,\nu^0_{\left(a_2-a_3\right)\left(H_1-1\right)-\left(a_1-a_3\right)\left(H_2-1\right)})
+N(r,\nu^\infty_f)
+o(T(r,f)+T(r,g)),$$
since the right-hand side does not explicitly contain the terms
$$N(r,\nu^\infty_{\left(a_2-a_3\right)\left(H_1-1\right)-\left(a_1-a_3\right)\left(H_2-1\right)})
+N(r,\nu^\infty_H).$$

The main subtlety in both arguments seems to come from the fact that the preimage sets of the small functions under $f$ and $g$ may fail to coincide on a certain exceptional set $S$; see, for instance, inequality (1) in the proof of \cite[Lemma 5]{LY1} and the inequality
$$N_{(3}(r,\nu^c_f)\le N_{(3}(r,\nu^0_H)$$
in \cite[Lemma 8]{LZ}. Although the corresponding truncated counting functions on $S$ are negligible, the counting functions with multiplicities taken into account may not necessarily remain small. For completeness, we therefore provide a detailed proof of the lemma below.

Our first purpose in this paper is to improve Lemma B to the case where two functions $f$ and $g$ share three small functions with bi-weights and provide a clear proof. To state the result, first of all, we give the following notation and definition. 

Let $f,a,S$ be as above and $n,k$ two positive integers, with $n\le k$. We define
$$E^{S}_{n,k)}(a,f)=\left\{\bigl(z,\nu^a_{f,\le n+1}(z)\bigl)\bigl |z\not\in S, \nu^a_{f}(z)\le k\right\}.$$
Note that $E^S_{n,+\infty)}(a,f)=E^S_{n}(a,f)$. We write $E^S_{n)}(a,f)$ for $E_{n,n)}(a,f)$.

\vskip0.2cm
\noindent
{\bf Definition C.} {\it Let $f$ and $g$ be two meromorphic functions on $\C$. Let $(a,b)$ be a pair of meromorphic functions on $\C$ and $n,k\ (n\le k)$ two positive integers or $+\infty$. We say that $f$ and $g$ weakly share $(a,b)$ with the bi-weight $(n,k)$ if there exists a discrete subset $S$ of $\overline\C$ of the counting function equal to $o(T(r,f)+T(r,g))$ such that $E^S_{n,k)}(a,f)=E^S_{n,k)}(b,g)$.}

\vskip0.1cm
% Similarly, we say that $f$ and $g$ weakly share the pair $(a,b)$ with the weight $n$ if $E^S_{n}(a,f)=E^S_{n}(b,g)$; weakly share the pair $(a,b)$ with the sub-weight $n$ if $E^{S}_{n)}(a,f)=E^{S}_{n)}(b,g)$. 
We will prove the following.

\begin{theorem}\label{1.1}  Let $f$ and $g$ be two non-constant meromorphic functions, $a_1, a_2$ and $a_3$ be three distinct small functions with respect to $f$ and $g$. Let $n_1,n_2,n_3$ and $k$ be positive integers or $+\infty$ with $n_i\le k\ (1\le i\le 3)$ satisfying $\Delta:=n_1n_2n_3-n_1-n_2-n_3-2>0$. If $f$ and $g$ weakly share $a_i\ (1\le i\le 3)$  with the bi-weights $(n_i,k)$, and if $f$ is not a quasi-M\"{o}bius transformation of $g$, then for any small function $c\ (\not\equiv a_1,a_2,a_3)$ with respect to $f$ and $g$, for every positive integer $n$, we have:
\begin{itemize}
\item[(a)] $\|\ N_{(3}(r,\nu^c_f)\le\left(\frac{2\lambda_n}{k+1}+\frac{6}{n+2}\right)T(r)+S(r),$
    \item[(b)] $\|\ N(r, \nu^c_f)\ge T(r,f)-\left(\frac{\lambda_n}{k+1}+\frac{2}{n+2}\right)T(r)+S(r),$
\end{itemize}
where $T(r)=T(r,f)+T(r,g)$ and $S(r)=o(T(r))$ and
$$\lambda_n=2^{n^2+2n+3}\left(\frac{2\sum_{1\le i<j\le 3}(n_i+1)(n_j+1)}{\Delta(k+1)}+\frac{3}{k+1}\right).$$
\end{theorem}
We note that for the function $H_1,H_2$ defined as above, the counting functions $\overline N(r,\nu^0_{H_i})$ and $\overline N(r,\nu^\infty_{H_i})$ are may not small term (i.e., may not equal to $o(T(r,f)+T(r,g))$). Therefore, most of arguments in the proof of Lemma B in \cite{LY2} cannot be applied in our situation. 

From Theorem\ref{1.1}, by letting $k\rightarrow +\infty$ and then letting $n\rightarrow +\infty$ we immediately get the following corollary.

\begin{corollary} Let $f$ and $g$ be two non-constant meromorphic functions, $a_1, a_2$ and $a_3$ be three distinct small functions with respect to $f$ and $g$. Let $n_1,n_2,n_3$ be positive integers satisfying $\Delta:=n_1n_2n_3-n_1-n_2-n_3-2>0$. If $f$ and $g$ weakly share $a_i$  with the weights $n_i$ for every $1\le i\le 3$, and if $f$ is not a quasi-M\"{o}bius transformation of $g$, then for any small function $c\ (\not\equiv a_1,a_2,a_3)$ with respect to $f$ and $g$, for every $\epsilon>0$, we have:
\begin{itemize}
\item[(a)] $\|\ N_{(3}(r,\nu^c_f)\le \epsilon T(r),$
    \item[(b)] $\|\ N(r, \nu^c_f)\ge (1- \epsilon)T(r,f),$
\end{itemize}
where $T(r)=T(r,f)+T(r,g)$.
\end{corollary}
Hence, this corollary is an imrovement of Lemma B. 
% We note that, since the preimage sets of the small functions under the functions $f$ and $g$ may not coincide on a certain set, most of arguments in the proof of Lemma B in \cite{LY1,LZ} cannot be applied in our situation.

Applying Theorem \ref{1.1}, we prove the following result about two meromorphic functions weakly share three pairs of small functions and satisfy an additional condition, which is weaker than the sharing condition for the fourth pair of small functions.
\begin{theorem}\label{1.2}
Let $f$ and $g$ be non-constant meromorphic functions and $a_i,b_i\ (i=1,2,3,4)$ $ (a_i\ne a_j, b_i\ne b_j, i\ne j)$ be small functions (with respect to $f$ and $g$).  Let $n_1,n_2,n_3$ and $k$ be positive integers such that $\Delta:=n_1n_2n_3-n_1-n_2-n_3-2>0$ and there exists a positive integer $n$ satisfying
$$\frac{48\sum_{1\le i<j\le 3}(n_i+1)(n_j+1)+72\Delta}{\Delta(k+1)}+\frac{30\lambda_n}{k+1}+\frac{88}{n+2}<1,$$
where $\lambda_n=2^{n^2+2n+3}\left(\frac{2\sum_{1\le i<j\le 3}(n_i+1)(n_j+1)}{\Delta(k+1)}+\frac{3}{k+1}\right)$.
If $f$ and $g$ weakly share the pair $(a_i,b_i)$ with bi-weight $(n_i,k)$ for every $i=1,2,3$ and satisfy a condition that
$$E^S_{1)}(a_4,f)=E^S_{1)}(b_4,g),E^S_{2)}(a_4,f)\subset E^S_{2}(b_4,g),E^S_{2)}(b_4,g)\subset E^S_{2}(a_4,f)$$
for a discrete subset $S$ of with $\| N(r,\nu_S)=o(T(r,f)+T(r,g))$, then $f$ and $g$ are quasi-M\"{o}bius transformation of each other.
\end{theorem}

Letting $k\longrightarrow +\infty$ and then letting $n\longrightarrow +\infty$, from Theorem \ref{2.2}, we get the following corollary about the case where two functions weakly share some pairs of small functions with weighted.
\begin{corollary}\label{1.3}
Let $f$ and $g$ be non-constant meromorphic functions and $a_i,b_i\ (i=1,2,3,4)$ $ (a_i\ne a_j, b_i\ne b_j, i\ne j)$ be small functions (with respect to $f$ and $g$).  Let $n_1,n_2,n_3$ be positive integers such that $\Delta:=n_1n_2n_3-n_1-n_2-n_3-2>0$.
If $f$ and $g$ weakly share the pair $(a_i,b_i)$ with the weight $n_i$ for every $i=1,2,3$ and satisfy
$$E^S_{1)}(a_4,f)=E^S_{1)}(b_4,g),E^S_{2)}(a_4,f)\subset E^S_{2}(b_4,g),E^S_{2)}(b_4,g)\subset E^S_{2}(a_4,f)\ \text{(*)}$$
for a discrete subset $S$ of small counting function (w.r.t $f$ and $g$), then $f$ and $g$ are quasi-M\"{o}bius transformation of each other.
\end{corollary}
We observe that condition $(*)$ is weaker than the condition
$$E^S_{2)}(a_4,f)=E^S_{2)}(b_4,g).$$
Therefore, if $f$ and $g$ share four pairs of small functions $(a_i,b_i)$ with weights $n_i\ (1\le i\le 4)$, then $f$ must be a quasi-M\"obius transformation of $g$ in each of the following cases:
\begin{enumerate}
\item $n_1=1,\ n_2=2,\ n_3=2,\ n_4=4;$
\item $n_1=n_2=n_3=2,\ n_4=3.$
\end{enumerate}
In particular, the first case recovers Theorem A of Li and Yang, while both cases extend the recent results of Quang-An \cite[Theorem 1.2 and Theorem 1.3]{QA}.
\section{Some lemmas and auxiliary results from Nevanlinna theory}

For a divisor $\nu$ on $\C$, we define the counting function of $\nu$ by
$$N(r,\nu)=\int\limits_1^r \dfrac {n(t)}{t}dt \quad (1<r<\infty), \text{ where } n(t)=\sum\limits_{|z|\leq t}\nu (z).$$
For two positive integers $k,M$ (maybe $M= \infty$), we set 
$$ \nu^{[M]}_{\leq k} (z)=
\begin{cases}
\min\{M,\nu (z)\}&\text{ if }\nu (z)\leq k\\
0&\text{ for otherwise. }
\end{cases}$$
and write $N^{[M]}_{k)}(r,\nu)$ for $N(r,\nu^{[M]}_{\leq k})$. We omit the character $^{[M]}$ (resp. $\leq k$) if $M=+\infty$ (resp. $k=+\infty$). In the same way, we define $\nu^{[M]}_{\ge k}$ and write $N^{[M]}_{(k}(r,\nu)$ and $N^{[M]}_{(\ell,k)}(r,\nu)$ for $N(r,\nu^{[M]}_{\ge k})$ and $N(r,(\nu_{\ge\ell})_{\le k}^{[M]})$, respectively. We also write $\overline N(r,\nu)$ for $N^{[1]}(r,\nu)$.

Let $f$ be a non-zero holomorphic function. For each $z_0\in\C$, expanding $f$ as
$ f(z) =\sum_{i=0}^\infty b_i(z-z_0)^i$ around $z_0$, then we define $\nu^{0}_{f}(z_0):=\min\{i\ :\ b_i \ne 0\}$. 

Let $\varphi$ be a non-constant meromorphic function. Then there are two holomorphic functions $\varphi_1,\varphi_2$ without common zeros such that $\varphi =\dfrac{\varphi_1}{\varphi_2}$. We define 
$ \nu_\varphi^0:=\nu_{\varphi_1}^0$ and $\nu_\varphi^\infty:=\nu_{\varphi_2}^0$ and $\nu_\varphi=\nu^0_\varphi -\nu^\infty_\varphi$. The proximity function of $\varphi$ is defined by:
$$ m(r,\varphi ):=\frac{1}{2\pi}\int\limits_{0}^{2\pi}\log^+|\varphi (re^{i\theta})|d\theta\ \ (r>1), $$
here $\log^+x=\max\{1,\log x\}$ for $x\in (0,\infty )$. The Nevanlinna characteristic function of $\varphi$ is defined by
$$ T(r,\varphi ):=m(r,\varphi ) +N(r,\nu^\infty_{\varphi}).$$

\begin{theorem}[\cite{Y}, Corollary 1]\label{2.1}
Let $f$ be a non-constant meromorphic function on $\C$. Let $a_1,\dots ,a_q\ (q\ge 3)$ be $q$ distinct small meromorphic functions (with respect to $f$) on $\C$. Then, for each $\epsilon >0$, the
following holds
$$\|\ (q-2-\epsilon)T(r,f)\le \sum_{i=1}^q\overline N(r,\nu^0_{f-a_i})+o(T(r,f)).$$
\end{theorem}

\begin{proposition}[\cite{QA}, Proposition 3.8]\label{2.2}
Let $f$ and $g$ be non-constant meromorphic functions and $a_i,b_i\ (i=1,2,3)$ $ (a_i\ne a_j, b_i\ne b_j, i\ne j)$ be small functions (with respect to $f$ and $g$). Assume that $f$ is not a quasi-M\"{o}bius transformation of $g$. Then for every positive integer $n$ we have the following inequality
$$\|\  N(r,\nu)\le \overline N(r,|\nu^0_{f-a_1}-\nu^0_{g-b_1}|)+\overline N(r,|\nu^0_{f-a_2}-\nu^0_{g-b_2}|)+S(r), $$
where $S(r)=o(T(r,f)+T(r,g))$ outside a finite Borel measure set of $[1,+\infty)$ and $\nu$ is the divisor defined by 
$\nu (z)=\max\{0,\min\{\nu^0_{f-a_3}(z),\nu^0_{g-b_3}(z)\}-1\}.$
\end{proposition}

\begin{lemma}[{Cf. \cite{LY95}}]\label{2.3}
Let $f_1,f_2, ...,, f_n$ be non-constant meromorphic functions satisfying $f_1+f_2+...+f_n\equiv 1$. Then, we have
$$\|\ T(r,f_i)\le\sum_{j=1}^{n}N^{[n-1]}(r,\nu^0_{f_j})+(n-2)\sum_{j=1}^{n}\overline N_{n-1}\left(r,\nu^{\infty}_{f_j}\right )+o(\sum_{j=1}^{n}T(r,f_j)),\; i=1,2.$$
\end{lemma}

\begin{lemma}[{Cf. \cite[Lemma 3.2]{QK}}]\label{2.4}
Let $f_1$ and $f_2$ be two non-constant meromorphic functions. If $(f_1^sf_2^t-1)$ is not identically zero for all integers $s$ and $t$ $(|s|+|t|>0)$ then for any positive integer $n$, one of the following assertion holds:\\
(i) $\overline N(r,1;f_1,f_2)\leq (2^{n(n+2)}-n(n+2)-1)\sum_{i=1,2}\left(\overline N(r,\nu^0_{f_i})+\overline N(r,\nu^\infty_{f_i})\right)+\dfrac{1}{n+2}T(r)+S(r);$\\
(ii) $\overline N(r,1;f_1,f_2)\leq (2^{(n+1)^2}+n^4+4n^3+n^2-6n-2)\sum_{i=1,2}\left(\overline N(r,\nu^0_{f_i})+\overline N(r,\nu^\infty_{f_i})\right)+S(r),$\\
where $\overline N(r,1;f_1,f_2)$ denotes the reduced counting function of $f_1$ and $f_2$ related to the common $1$-points, $T(r)=T(r,f_1)+T(r,f_2)$ and $\|\ S(r)=o(T(r))$.
\end{lemma}

\begin{lemma}
Let $f$ and $g$ be non-constant meromorphic functions, and $n_1,n_2,n_3,k$ positive integers such that $k> n_i\ (\forall 1\le i\le 3)$ and $\Delta=n_1n_2n_3-n_1-n_2-n_3-2>0$. Suppose that $f$ and $g$ weakly share each $0,\infty,1$ with bi-weights $(n_1,k),(n_2,k),(n_3,k)$ respectively. If $f$ is not a M\"{o}bius transformation of $g$ then
$$\|\ \Delta\overline N_{(n_1+1,k)}(r,\nu^0_f)\le \frac{2(n_2+1)(n_3+1)}{k+1}T(r)+S(r),$$
where $T(r)=T(r,f)+T(r,g)$, $S(r)=o(T(r))$.
\end{lemma}
\begin{proof}
Denote by $S$ the discrete subset of $\C$ with $\|\ N(r,\nu_S)=o(T(r,f)+T(r,g))$ such that $E^S_{n_i,k)}(a_i,f)=E^S_{n_i,k)}(a_i,g)\ (i=1,2,3),$ where $a_1=0,a_2=\infty,a_3=1$. We set
\begin{align*}
H=\dfrac{f'}{f-1}-\dfrac{g'}{g-1}
\end{align*}
Since $f$ is not a M\"{o}bius transformation of $g$, $H\not\equiv 0$. Let $z\ (z\not\in S)$ be a zero of $f$ with $\nu^0_f(z)\le k$. If $\nu^0_f(z)> n_1$ then $z$ must be a zero of $H$ with multiplicity at least $n_1$. Therefore, we have
\begin{align}\label{3.6}
\begin{split}
n_1\overline N_{(n_1+1,k)}(r,\nu^{a_1}_f)&\le N(r,\nu^0_{H})\le T(r,H)=N(r,\nu^{\infty}_{H})+S(r)\\
&\leq \overline N(r,|\nu^{\infty}_f-\nu^{\infty}_g |)+\overline N(r,|\nu^{1}_f-\nu^{1}_g |)+S(r)\\
&\leq \sum_{i=2,3}(\overline N_{(n_i+1,k)}(r,\nu^{a_i}_{f})+\sum_{h=f,g}\overline N_{(k+1}(r,\nu^{a_i}_{h}))+S(r)\\
&\leq \sum_{i=2,3}\overline N_{(n_i+1,k)}(r,\nu^{a_i}_{f})+\dfrac{2}{k+1}T(r)+S(r).
\end{split}
\end{align}
Here, the third inequality comes from Proposition \ref{2.2}. Because $a_1,a_2,a_3$ play the same role, we have
$$n_2\overline N_{(n_2+1,k)}(r,\nu^{a_2}_f)\leq\sum_{i=1,3} \overline N_{(n_i+1,k)}(r,\nu^{a_i}_{f})+\dfrac{2}{k+1}T(r)+S(r).$$
Combining (\ref{3.6}) and the above inequality, we get
\begin{align*}
n_1n_2\overline N_{(n_1+1,k)}(r,\nu^{a_1}_f)\le&(n_2+1)\overline N_{(n_3+1,k)}(r,\nu^{a_3}_{f})\\
&+\overline N_{(n_1+1,k)}(r,\nu^{a_1}_{f})+\dfrac{2n_2+2}{k+1}T(r)+S(r).
\end{align*}
Thus,
\begin{align*}
(n_1n_2-1)\overline N_{(n_1+1,k)}(r,\nu^{a_1}_f)\le (n_2+1)\left(\overline N_{(n_3+1,k)}(r,\nu^{a_3}_{f})+\dfrac{2}{k+1}T(r)\right)+S(r).
\end{align*}
Similarly, we have
\begin{align*}
(n_3n_2-1)\overline N_{(n_3+1,k)}(r,\nu^{a_3}_f))\le (n_2+1)\left(\overline N_{(n_1+1,k)}(r,\nu^{a_1}_{f})+\dfrac{2}{k+1}T(r)\right)+S(r).
\end{align*}
From above two inequalities, we obtain
\begin{align*}
(n_3n_2-1)(n_1n_2-1)&\overline N_{(n_1+1,k)}(r,\nu^{a_1}_f)\le (n_3n_2-1)\dfrac{2(n_2+1)}{k+1}T(r)\\ 
&+(n_2+1)^2\left(\overline N_{(n_1+1,k)}(r,\nu^{a_1}_{f})+\dfrac{2}{k+1}T(r)\right)+S(r).
\end{align*}
This implies that
$$\Delta\overline N_{(n_1+1,k)}(r,\nu^{a_1}_f)\le \frac{2(n_2+1)(n_3+1)}{k+1}T(r)+S(r).$$
The lemma is proved.
\end{proof}
In the above lemma, if each $a_i$ is replaced by a pair of small function $(a_i,b_i)\ (1\le i\le 3),$ where $a_i\ne a_j,b_i\ne b_j$ whenever $i\ne j$, then by considering meromorphic functions $\dfrac{f-a_1}{f-a_2}\cdot \dfrac{a_3-a_2}{a_3-a_1}$ and $\dfrac{g-b_1}{g-b_2}\cdot \dfrac{b_3-b_2}{b_3-b_1}$ instead of $f$ and $g$, we will get the following estimate
\begin{align}\label{3.5}
\|\ \Delta\overline N_{(n_1+1,k)}(r,\nu^{a_1}_f)\le \frac{2(n_2+1)(n_3+1)}{k+1}T(r)+S(r),
\end{align}
provided that $f$ is not a quasi-M\"{o}bius transformation of $g$.

\begin{lemma}\label{3.7}
Let $f_0, f_1, f_2$ be non-constant meromorphic functions such that $f_0+f_1+f_2=1$. Let $a,b$ be meromorphic functions such that $af_1+bf_2\not\equiv 0$. If $f_0,f_1,f_2$ are linearly independent then
\begin{align*}
\biggl\|\ T\left(r,\frac{f_0}{af_1+bf_2}\right)&\le \sum_{0\le i\le 2}N^{[2]}(r,\nu^0_{f_i})+3\overline N\left(r,\sum_{0\le i\le 2}\nu^\infty_{f_i}\right)\\
&+o(\sum_{i=1}^nT(r,f_i))+O(T(r,a)+T(r,b)).
\end{align*}
\end{lemma}
\begin{proof}
Take a holomorphic function $h$ such that $\nu_h(z)=\max_{0\le i\le 2}\nu^\infty_{f_i}(z)\ \forall z\in\C,$ and 
consider the holomorphic curve $f$ from $\C$ into $\P^2(\C)$ with the reduced representation
$$ \tilde f=(hf_0,hf_1,hf_2).$$ 
Then, we have
$$W(\tilde f)=\det\left((hf_i)^{(j)};0\le i,j\le 2\right)=h^3\det\left(f_i^{(j)};0\le i,j\le 2\right)\not\equiv 0.$$
By the second main theorem in Nevanlinna theory, we have
\begin{align}\label{3.8}
\|\ T(r,f)\le\sum_{0\le i\le 2}N(r,\nu^0_{hf_i})+N(r,\nu^0_{\sum_{i=0}^2hf_i})-N(r,\nu^0_{W(\tilde f)})+o(T(r,f)).
\end{align}
Note that $\sum_{i=0}^2hf_i=h$. 

We set $\nu=\sum_{i=0}^2\nu^0_{hf_i}+\nu^0_h-\nu^0_{W(\tilde f)}$. For $z\in\C$, we consider the following cases.

Case 1: $z$ is not a pole of any $f_i$. Then we have $\nu^0_h(z)=0$ and hence
$$\nu (z)\le \sum_{0\le i\le 2}\nu^0_{f_i}(z)-\nu^0_{\det\left(f_i^{(j)};0\le i,j\le 2\right)}(z)\le\sum_{0\le i\le 2}\min\{2,\nu^0_{f_i}(z)\}.$$
 
Case 2: $z$ is a pole of some $f_i$. We suppose that $\nu^\infty_{f_0}(z)\ge\nu^\infty_{f_1}(z)\ge\nu^\infty_{f_2}(z)$. Then, we must have
$ \nu^\infty_{f_0}(z)=\nu^\infty_{f_1}(z)>0,$ and hence
\begin{align*}
\nu (z)&\le \nu^0_{hf_2}(z)+\nu^0_{h}(z)-\nu^0_{\det\left((hf_i)^{(j)};1\le i\le 3,0\le j\le 2\right)}(z)\ \ \ \ \ \text{ [where $f_3=1$]}\\
&\le \nu^0_{hf_2}(z)+\nu^0_{h}(z)-\min_{0\le s,t\le 2,s\ne t}\{\nu^0_{(hf_2)^{(s)}}(z)+\nu^0_{(hf_3)^{(t)}}(z)\}\\
&\le \nu^0_{hf_2}(z)+\nu^0_{h}(z)-\min_{0\le s,t\le 2,s\ne t}\{\max\{0,\nu^0_{hf_2}(z)-s\}+\max\{0,\nu^0_{h}(z)-t\}\}\\
&\le \max_{0\le s,t\le 2,s\ne t}\{\min\{\nu^0_{hf_2}(z),s\}+\min\{\nu^0_{h}(z),t\}\}\le 3.
\end{align*}
This implies that 
$$\nu(z)\le 3\min\{1,\sum_{0\le i\le 2}\nu^\infty_{f_i}(z)\}.$$

Then, we obtain
$$\nu\le \sum_{i=0}^2\min\left\{2,\nu^0_{f_i}\right\}+3\min\left\{1,\sum_{i=0}^2\nu^\infty_{f_i}\right\}.$$
This implies that
$$ N(r,\nu)\le\sum_{i=0}^2N^{[2]}(r,\nu^0_{f_i})+3\overline N\left(r,\sum_{i=0}^2\nu^\infty_{f_i}\right).$$
Combining this inequality with (\ref{3.8}), we get
\begin{align}\label{3.9}
\|\ T(r,f)\le \sum_{i=0}^2N^{[2]}(r,\nu^0_{f_i})+3\overline N(r,\sum_{i=0}^2\nu^\infty_{f_i})+o(\sum_{i=1}^nT(r,f_i)). 
\end{align}
Now, we take a meromorphic function $\varphi$ such that $\varphi hf_0,\varphi ahf_1,\varphi bhf_2$ are holomorphic functions without common zero and set $F=(\varphi hf_0,\varphi ahf_1,\varphi bhf_2)$. By the definition of the characteristic function and Jensen's formula, we have
\begin{align*}
\biggl\|\ T\left(r,\frac{f_0}{af_1+bf_2}\right)&\le\frac{1}{2\pi}\int_0^{2\pi}\log\|F(re^{i\theta})\|d\theta\\
&\le \frac{1}{2\pi}\int_0^{2\pi}\left(\log\|\tilde f(re^{i\theta})\|+\log^+|a|+\log^+|b|+\log |\varphi|\right)d\theta\\
&\le T(r,f)+m(r,a)+m(r,b)+N(r,\nu^0_\varphi)-N(r,\nu^\infty_\varphi)+O(1)\\
&=T(r,f)+O(T(r,a)+T(r,b))\\
&\le \sum_{i=0}^2N^{[2]}(r,\nu^0_{f_i})+3\overline N(r,\sum_{i=0}^2\nu^\infty_{f_i})\ \ \ \text{\small[by (\ref{3.9})]}\\
&+o(\sum_{i=1}^nT(r,f_i))+O(T(r,a)+T(r,b)).
\end{align*}
The lemma is proved.
\end{proof}
%_____________________________________________%
\section{Proof of Main Theorem}

\begin{proof}[Proof of Theorem \ref{1.1}]
By using a quasi-M\"obius transformation if necessary, without loss of generality, we may assume that $a_1=b_1=0,a_2=b_2=\infty, a_3=b_3=1$. Denote by $S$ the discrete subset of $\overline\C$ with $\|\ N(r,\nu_S)=o(T(r,f)+T(r,g))$ such that $E^S_{n_i,k)}(a_i,f)=E^S_{n_i,k)}(b_i,g)\ (1\le i\le 3)$. We set $h_1=\frac{g}{f},h_2=\frac{g-1}{f-1},f_1=ch_1,f_2=(1-c)h_2$ and $f_0=1-f_1-f_2$. Thus
$$f_0+f_1+f_2\equiv 1.$$ 
It follows from the supposition that none of $f_0,f_1,f_2$ is constant. We consider the following two cases.

If $f_0,f_1,f_2$ are linearly dependent over $\mathbb{C}$, then there exist three constants $c_0,c_1,c_2$ (not all zeros) such that
$$c_0f_0+c_1f_1+c_2f_2=0.$$
Obviously, $c_0\neq 0$, otherwise $\frac{h_2}{h_1}$ is a small function and hence $f$ will be a quasi-M\"{o}bius transformation $g$. Then we have
$$\left(1-\dfrac{c_1}{c_0}\right)f_1+\left(1-\dfrac{c_2}{c_0}\right)f_2\equiv 1.$$
Since both $f_1$ and $f_2$ are not constant so $c_0\ne c_1$ and $c_0\ne c_2$. Therefore, by Theorem \ref{2.1} we have
$$\|\ T(r,f_1)\le \sum_{i=1,2}\left(\overline N\left(r,\nu^0_{f_i}\right)+\overline N\left(r,\nu^\infty_{f_i}\right)\right)+S(r),$$
where $S(r)=o(T(r,f)+T(r,g))$. This implies that
$$\|\ T(r,h_1)\le 2\sum_{i=1,2}\left(\overline N\left(r,\nu^0_{h_i}\right)+\overline N\left(r,\nu^\infty_{h_i}\right)\right)+S(r).$$
In the same way, we have 
$$\|\ T(r,h_2)\le 2\sum_{i=1,2}\left(\overline N\left(r,\nu^0_{h_i}\right)+\overline N\left(r,\nu^\infty_{h_i}\right)\right)+S(r).$$
On the other hand, since $f=\frac{1-h_2}{h_1-h_2}$ and $g=\frac{1-h_2}{1-h_2/h_1}=\frac{(1/h_2)-1}{(1/h_2)-(1/h_1)}$, we have
\begin{align*}
&T(r,f)\le T(r,h_1)+T(r,h_2)+O(1) \\ 
\text{ and }\ & T(r,g)\le T\left(r,\frac{1}{h_1}\right)+T\left(r,\frac{1}{h_1}\right)+O(1)=T(r,h_1)+T(r,h_2)+O(1).
\end{align*}
Hence, 
\begin{align}\label{3.10}
\begin{split}
\|\ T(r,f)&+T(r,g)\le 8\sum_{i=1,2}\left(\overline N(r,\nu^0_{h_i})+\overline N(r,\nu^\infty_{h_i})\right)+S(r)\\ 
&\le 8\sum_{i=1}^3\overline N(r,|\nu^{a_i}_f-\nu^{b_i}_g|)+S(r)\\
&\le 8\sum_{i=1}^3\overline N_{(n_i+1,k)}(r,\nu^{a_i}_f)+8\sum_{i=1}^3\left(\overline N_{(k+1}(r,\nu^{a_i}_f)+\overline N_{(k+1}(r,\nu^{b_i}_g) \right)+S(r)\\
&\le \left(\frac{16\sum_{1\le i<j\le 3}(n_i+1)(n_j+1)}{\Delta(k+1)}+\frac{24}{k+1}\right)T(r)+S(r).
\end{split}
\end{align}
Here, we get the last inequality from  (\ref{3.5}). Letting $r\rightarrow +\infty$, we get 
$$ k+1\le 24+\frac{16\sum_{1\le i<j\le 3}(n_i+1)(n_j+1)}{\Delta}.$$
This is a contradiction. Therefore, $f_0,f_1,f_2$ must be linearly independent. 

(a) We prove the assertion (a) of the theorem. Since 
$$f-c=\frac{1-ch_1+(c-1)h_2}{h_1-h_2}=\frac{f_0}{\frac{1}{c}f_1+\frac{1}{c-1}f_2},$$ by Lemma \ref{3.7}, we have
\begin{align}\label{3.11}
\begin{split}
\|\ T(r,f-c)&\le\sum_{i=0}^{2}N^{[2]}\left(r,\nu^0_{f_i}\right)+3\overline N(r,\sum_{i=0}^2\nu^\infty_{f_i})+S(r)\\
&\leq N^{[2]}\left(r,\nu^0_{f_0}\right)+2\sum_{i=1,2}\overline N\left(r,\nu^0_{h_i}\right)+3\sum_{i=1,2}\overline N\left(r,\nu^\infty_{h_i}\right)+S(r)\\
&\leq N^{[2]}\left(r,\nu^0_{f_0}\right)+3\sum_{i=1}^3\overline N\left(r,|\nu^{a_i}_f-\nu^{b_i}_g|\right)+S(r).
\end{split}
\end{align}
We note that $f_0=(f-c)(h_1-h_2)$. Suppose that $z_0\ (z_0\not\in S)$ is a zero of $f_0$. We consider the following three cases.
\begin{itemize}
\item Suppose that $z_0$ is not a zero of $h_1-h_2$. Then we have
$$\min\{\nu^0_{f_0}(z_0),2\}\le \min\{\nu^0_{f-c}(z_0),2\}.$$
\item Suppose that $z_0$ is a zero of $h_1-h_2$ and is a pole of $c$. Then we have
$$\min\{\nu^0_{f_0}(z_0),2\}\le 2\min\{\nu^\infty_{c}(z_0),1\}.$$
\item Suppose that $z_0$ is a zero of $h_1-h_2$ and is not a pole of $c$. Then  $z_0$ is a zero of $1-h_2$ and hence a common 1-point of $h_1$ and $h_2$.
\end{itemize}
Therefore, from the above three cases we arrive that
\begin{align}
\label{3.12}
N^{[2]}\left(r,\nu^0_{f_0}\right)\le 2\overline N^{[2]}(r,\nu^0_{f-c})+2\overline N(r,1;h_1,h_2)+S(r).
\end{align}

Combining (\ref{3.11}) and (\ref{3.12}), we get
\begin{align}\label{3.13}
\begin{split}
\|\ N(r,\nu^0_{f-c})&\le T(r,f-c)+S(r)\\
&\le 2\overline N^{[2]}(r,\nu^0_{f-c})+2\overline N(r,1;h_1,h_2)+3\sum_{i=1}^3\overline N\left(r,|\nu^{a_i}_f-\nu^{b_i}_g|\right)+S(r).
\end{split}
\end{align}
We note that 
$$N(r,\nu^0_{f-c})=N_{2)}(r,\nu^0_{f-c})+N_{(3}(r,\nu^0_{f-c})\ge N^{[2]}(r,\nu^0_{f-c})+\frac{1}{3}N_{(3}(r,\nu^0_{f-c}).$$
Therefore, from (\ref{3.13}) we have
\begin{align}\label{3.14}
\|\  N_{(3}(r,\nu^0_{f-c})\le 6\overline N(r,1;h_1,h_2)+9\sum_{i=1}^3\overline N\left(r,|\nu^{a_i}_f-\nu^{b_i}_g|\right)+S(r).  
\end{align}

Suppose that there exist two integers $s$ and $t\ (|s|+|t|>0)$ such that $h_1^{s}=h_2^{t}$, i.e, $g^s(g-1)^t=f^s(f-1)^t$. If $s=0$ or $t=0$ or $s=-t$ then it is easy to see that $f$ and $g$ are M\"{o}bius transformation of each other. Therefore, we may assume that $s\ne 0,t\ne 0,s\ne -t$. 

For a point $z_0\in\C$, we note that:
\begin{itemize}
\item If $z_0$ is a common $\alpha$-point of $f$ and $g$  with $\alpha\in\{0,1,\infty\}$ then $\nu^\alpha_f(z_0)=\nu^\alpha_g(z_0)$ and hence $z_0$ is not a pole of $h_1$ or $h_2$. 

\item $f(z_0)\in\{0,1,\infty\}$ if and only if $g(z_0)\in\{0,1,\infty\}$ and in this case, we have $$\nu^{g(z_0)}_g(z_0)\le (|s|+|t|)\nu^{f(z_0)}_f(z_0).$$
\end{itemize}
Therefore, if $z_0$ is a zero of $f-c$ and also is a pole of $h_1-h_2$ then $z_0$ must be either a zero or an $1$-point of $f$ and one has $\nu^0_{f-c}(z_0)\le\nu^0_{c-f(z_0)}(z_0)$ or $\nu^0_{f-f(z_0)}(z_0)\le\nu^0_{c-f(z_0)}(z_0)$. The later inequality implies that
$$\nu^\infty_{h_1-h_2}(z_0)\le (1+|s|+|t|)\nu^{f(z_0)}_f(z_0)=(1+|s|+|t|)\nu^{0}_{f-f(z_0)}(z_0)\le (1+|s|+|t|)\nu^{0}_{c-f(z_0)}(z_0).$$
This implies that
$$\nu^0_{f-c,\ge 3}(z_0)\le \nu^0_{f_0,\ge 3}(z_0)+2(1+|s|+|t|)\nu^{0}_{c-f(z_0)}(z_0).$$
Of course, if $z_0$ is not a pole of $h_1-h_2$ then $$\nu^0_{f-c,\ge 3}(z_0)\le \nu^0_{f_0,\ge 3}(z_0).$$ 
Hence, we have
$$\nu^0_{f-c,\ge 3}\le \nu^0_{f_0,\ge 3}+2(1+|s|+|t|)\sum_{\alpha=0,1,\infty}\nu^{\alpha}_{c}.$$
This yields that
$$\|\ N_{(3}\left(r,\nu^0_{f-c}\right)\le N_{(3}\left(r,\nu^0_{f_0}\right)+S(r).$$
Similarly as (\ref{3.11}), we have 
\begin{align*}
\|\ N(r,\nu^0_{f_0})&\le T(r,f_0)\\
&\le N^{[2]}\left(r,\nu^0_{f_0}\right)+3\sum_{i=1}^3\overline N\left(r,|\nu^{a_i}_f-\nu^{b_i}_g|\right)+S(r)\\
&=N^{[2]}\left(r,\nu^0_{f_0}\right)+3\left(\frac{2\sum_{1\le i<j\le 3}(n_i+1)(n_j+1)}{\Delta(k+1)}+\frac{3}{k+1}\right)T(r) +S(r).
\end{align*}
This implies that $\|\ N_{(3}\left(r,\nu^0_{f_0}\right)\le 3\left(\frac{2\sum_{1\le i<j\le 3}(n_i+1)(n_j+1)}{\Delta(k+1)}+\frac{3}{k+1}\right)T(r) +S(r)$ and hence
$$\|\ N_{(3}\left(r,\nu^0_{f-c}\right)\le N_{(3}\left(r,\nu^0_{f_0}\right)\le 3\left(\frac{2\sum_{1\le i<j\le 3}(n_i+1)(n_j+1)}{\Delta(k+1)}+\frac{3}{k+1}\right)T(r) +S(r).$$
Therefore, we get the assertion (a) in this case.

Now, we consider the remaining case where $h_1^{s}\not\equiv h_2^{t}$ for every integers $s,t\ (|t|+|s|>0)$. Note that $T(r,h_i)\le T(r)\ (i=1,2)$ and
\begin{align}\label{3.16}
\begin{split}
\sum_{i=1}^2(\overline N(r,\nu^0_{h_i})&+\overline N(r,\nu^\infty_{h_i})\le 2\sum_{i=1}^3\overline N(r,|\nu^{a_i}_f-\nu^{b_i}_g|)+S(r),\\ 
\text{\small{[similar as (\ref{3.10})]}}\ \ &\le 2\left(\frac{2\sum_{1\le i<j\le 3}(n_i+1)(n_j+1)}{\Delta(k+1)}+\frac{3}{k+1}\right)T(r)+S(r)\\
\end{split}
\end{align}
Then, by (\ref{3.14}-\ref{3.16}) and Lemma \ref{2.4}, for every positive integer $n\ge 2$, we have
\begin{align*}
\|\ N_{(3}\left(r,\nu^0_{f-c}\right)&\le\left((9+6\delta)\left(\frac{2\sum_{1\le i<j\le 3}(n_i+1)(n_j+1)}{\Delta(k+1)}+\frac{3}{k+1}\right)+\frac{6}{n+2}\right)T(r)+S(r),
\end{align*}
where $\delta=2^{n^2+2n+1}+n^4+4n^3+n^2-6n-2.$ Thus,
\begin{align*}
\|\ N_{(3}\left(r,\nu^0_{f-c}\right)&\le\left(2^{n^2+2n+4}\left(\frac{2\sum_{1\le i<j\le 3}(n_i+1)(n_j+1)}{\Delta(k+1)}+\frac{3}{k+1}\right)+\frac{6}{n+2}\right)T(r)+S(r)\\
&=\left(\frac{2\lambda_n}{k+1}+\frac{6}{n+2}\right)T(r)+S(r).
\end{align*}
The assertion (a) of the theorem is proved.

(b) We prove the assertion (b) of the theorem. It is deduced from (\ref{3.11}) and (\ref{3.12}) that
\begin{align}\label{3.17}
T(r,f)\le N\left(r,\nu^0_{f-c}\right)+2\overline N(r,1;h_1,h_2)+3\sum_{i=1}^3\overline N\left(r,|\nu^{a_i}_f-\nu^{b_i}_g|\right)+S(r).
\end{align}
By Lemma \ref{2.4} and similarly as (\ref{3.16}), we have
\begin{align*}
\|\ &2\overline N(r,1;h_1,h_2)+3\sum_{i=1}^3\overline N\left(r,|\nu^{a_i}_f-\nu^{b_i}_g|\right)\\
&\le\left((3+2\delta)\left(\frac{2\sum_{1\le i<j\le 3}(n_i+1)(n_j+1)}{\Delta(k+1)}+\frac{3}{k+1}\right)+\frac{2}{n+2}\right)T(r)+S(r)\\
&\le \left(2^{n^2+2n+3}\left(\frac{2\sum_{1\le i<j\le 3}(n_i+1)(n_j+1)}{\Delta(k+1)}+\frac{3}{k+1}\right)+\frac{2}{n+2}\right)T(r)+S(r)\\
&=\left(\frac{\lambda_n}{k+1}+\frac{2}{n+2}\right)T(r)+S(r).
\end{align*}
Combining the above inequality and (\ref{3.17}), we have
$$\|\ N\left(r,\nu^0_{f-c}\right)\ge T(r,f)-\left(\frac{\lambda_n}{k+1}+\frac{2}{n+2}\right)T(r)+S(r).$$
The assertion (b) is proved.
\end{proof}

\begin{proof}[Proof of Theorem \ref{1.2}]
By using a quasi-M\"{o}bius transformation if necessary, we may suppose that $a_1=b_1=0,a_2=b_2=\infty,a_3=b_3=1,a_4=a,b_4=b$ and $a,b\not\in\{0,1,\infty\}$. Suppose by contrary that $f$ is not a quasi-M\"{o}bius transformation of $g$. We set $T(r)=T(r,f)+T(r,g)$ and $S(r)=S(r,f)+S(r,g).$ 
We define divisors 
$$\mu_i(z)=\min\{1,|\nu^0_{f-a_i}(z)-\nu^0_{g-b_i}(z)|\}\text{ for }i=1,2,3\text{ and }\mu_4(z)=|\nu^0_{f-a_4}(z)-\nu^0_{g-b_4}(z)|.$$
By the supposition and by (\ref{3.5}), we have
\begin{align*}
\|\ \Delta\sum_{i=1}^3N(r,\mu_i)&\le\Delta\sum_{i=1}^3\left(\overline N_{(n_i+1,k)}(r,\nu^{a_i}_f)+\overline N_{(k+1}(r,\nu^{a_i}_f)+\overline N_{(k+1}(r,\nu^{b_i}_g)\right)+S(r)\\
&\le \frac{2\sum_{1\le i<j\le 3}(n_i+1)(n_j+1)+3\Delta}{k+1}T(r)+S(r).
\end{align*}
Also, by the main theorem, we have
$$\|\ N(r,\mu_4)\le N_{(3}(r,\nu^{a_4}_f)+N_{(3}(r,\nu^{b_4}_g)+S(r)\le 2\left(\frac{2\lambda_n}{k+1}+\frac{6}{n+2}\right)T(r)+S(r),$$
where
$$\lambda_n=2^{n^2+2n+3}\left(\frac{2\sum_{1\le i<j\le 3}(n_i+1)(n_j+1)}{\Delta(k+1)}+\frac{3}{k+1}\right)$$
for a positive integer $n$ (chosen later). Combining the above two inequalities, we obtain
\begin{align}\label{4.2}
\begin{split}
4&\sum_{i=1}^3N(r,\mu_i)+N(r,\mu_4)\\
&\le \left(\frac{8\sum_{1\le i<j\le 3}(n_i+1)(n_j+1)+12\Delta}{\Delta(k+1)}+\frac{4\lambda_n}{k+1}+\frac{12}{n+2}\right)T(r)+S(r).
\end{split}
\end{align}
We write $f=\frac{f_0}{f_1}$ (resp. $g=\frac{g_0}{g_1}$), where $f_0,f_1$ (resp. $g_0,g_1$) are holomorphic function without common zeros. Take meromorphic functions $h_1,h_2,h_3,h_4$ such that
$$f_0=h_1 g_0;f_1=h_2 g_1;f_0-f_1=h_3(g_0-g_1);f_0-af_1=h_4(g_0-bg_1).$$
It is easy to see that 
$$\nu^0_{h_i}+\nu^\infty_{h_i}=|\nu^{a_i}_f-\nu^{b_i}_g|\ (1\le i\le 4)$$
outside a discrete subset $S$ of $\C$ with $\|\ N(r,\nu_S)=S(r)$.
From the definition of functions $h_i\ (1\le i\le 4)$, we have 
\begin{align*}
\mathrm{det}\left (
\begin{array}{cccc}
1&0&-h_1&0\\ 
0&1&0&-h_2\\
1&-1&-h_3&h_3\\
1&-a&-h_4&bh_4
\end{array}\right )=0.
\end{align*}
Then
\begin{align*}
%\label{3.12}
(1-a)h_1h_2-h_1h_3+bh_1h_4+a h_2h_3-h_2h_4+(1-b) h_3h_4=0.
\end{align*}
We define the following functions:
\begin{align*}
h_{1,2}&=(1-a)h_1h_2,\ h_{1,3}=-h_1h_3,\ h_{1,4}=b h_1h_4,\\
h_{2,3}&=a h_2h_3,\ h_{2,4}=-h_2h_4,\ h_{3,4}=(1-b) h_3h_4.
\end{align*}
Then we have 
$$ \sum_{1\le i<j\le 4}h_{i,j}=0. $$
Let $d$ be a meromorphic function on $\C$ such that $dh_{i,j}\ (1\le i<j\le 4)$ are all holomorphic functions on $\C$ without common zero. Then it is easy to see that
\begin{align*}
 \sum_{1\le i<j\le 4}N^{[4]}(r,dh_{i,j})\le& 3\sum_{i=1}^4(N^{[4]}(r,\nu^0_{h_i})+N^{[4]}(r,\nu^\infty_{h_i}))+S(r)\\ 
\le& 12\sum_{i=1}^3N(r,\mu_i)+3N(r,\mu_4)+S(r).
\end{align*}
Let $\mathcal I=\{(i,j)|1\le i<j\le 4\}$. For $I=(i,j)\in\mathcal I$, we denote $h_{i,j}$ by $h_{I}$.
Take $I_0\in\mathcal I$ arbitrarily. Then 
$$dh_{I_0}=-\sum_{I\ne I_0}dh_I.$$
Denote by $t$ the minimum number satisfying the following: There exist $t$ elements $I_1$,..., $I_t \in\mathcal I$ and $t$ nonzero constants $b_v\in\C\ (1\le v\le t)$ such that $dh_{I_0}=\sum_{v=1}^tb_vdh_{I_v}.$

By the minimality of $t$, the family $\{dh_{I_1},...,dh_{I_t}\}$ is linearly independent over $\C$.

Suppose that $t\ge 2$. Consider the linearly non-degenerate holomorphic mapping $h:\C \to \P^{t-1}(\C)$ with the representation $h=(d h_{I_1}:...:d h_{I_t})$. By the second main theorem in Nevanlinna theory for hyperplanes, we have
\begin{align}\label{4.4}
\begin{split}
\|\ T_h(r)&\le \sum_{v=1}^t N^{[t-1]}_{dh_{I_v}}(r)+N_{dh_{I_0}}^{[t-1]}(r)+S(r)\\
&\le \sum_{v=1}^6 N^{[4]}_{dh_{I_v}}(r)+S(r)\ \  [\text{since $t\le 5$}]\\
&\le 12\sum_{i=1}^3N(r,\mu_i)+3N(r,\mu_4)+S(r).
\end{split}
\end{align}
On the other hand, there exists two element, say $I_i$ and $I_j$, such that $4\not\in (I_i\setminus I_j)\cup (I_j\setminus I_i)$. Therefore, by the main theorem we have
\begin{align}\label{4.5}
\begin{split}
\|\ T_h(r)&\ge T(r,\frac{h_{I_i}}{h_{I_j}})+S(r)=T(r,\frac{h_{I_i}}{h_{I_j}}-\frac{\tilde h_{I_i}}{\tilde h_{I_j}})+S(r)\\
&\ge N(r,\min\{\nu^0_{f-a},\nu^0_{g-b}\})+S(r)\\
&\ge \frac{1}{2}\left(N(r,\nu^a_{f})+N(r,\nu^a_{f})-N_{(3}(r,\nu^a_{f})-N_{(3}(r,\nu^a_{f})\right)+S(r)\\
&\ge \frac{1}{2}\left(1-\frac{6\lambda_n}{k+1}-\frac{16}{n+2}\right)T(r)+S(r),
\end{split}
\end{align}
where $\tilde h_{I_i},\tilde h_{I_j}$ are obtained from $h_{I_i},h_{I_j}$ by substituting $a,b$ into $f,g$ respectively.
From (\ref{4.2}), (\ref{4.4}) and (\ref{4.5}), we have
\begin{align*}
&\biggl\|\  \frac{1}{2}\left(1-\frac{6\lambda_n}{k+1}-\frac{16}{n+2}\right)T(r)\\ 
& \le 3\left(\frac{8\sum_{1\le i<j\le 3}(n_i+1)(n_j+1)+12\Delta}{\Delta(k+1)}+\frac{4\lambda_n}{k+1}+\frac{12}{n+2}\right)T(r)+S(r).
\end{align*}
Letting $r\longrightarrow +\infty$, we get 
$$1\le \frac{48\sum_{1\le i<j\le 3}(n_i+1)(n_j+1)+72\Delta}{\Delta(k+1)}+\frac{30\lambda_n}{k+1}+\frac{88}{n+2}.$$
This is a contradiction.

Therefore, we must have $t=1$, i.e., there is $I_1\in \mathcal {I}\setminus \{I_0\}$ such that  $\dfrac {h_{I_0}}{h_{I_1}}\in\C\setminus\{0\}$. Hence, for every $I\in\mathcal J$, there exists $J\in\mathcal I\setminus\{I\}$ such that $\dfrac{h_I}{h_J}\in\C\setminus\{0\}$. We consider the following three cases:

\textbf{Case a:} There are $I=(s,t), J=(s,\ell)$ such that $\dfrac{h_I}{h_J}\in\C\setminus\{0\}$. Then $h_t=\alpha h_\ell$ with a small meromorphic function (w.r.t $f$ and $g$) $\alpha$. Therefore, $f$ is a quasi-M\"{o}bius  transformation of $g$. This contradicts the supposition.

\textbf{Case b:} There are $I=(t,s), J=(\ell,s)$ such that $\dfrac{h_I}{h_J}\in\C\setminus\{0\}$.  Similarly as the above case, $f$ is a quasi-M\"{o}bius  transformation of $g$ and we arrive at the contradiction. 

\textbf{Case c:} There are small functions $\alpha,\beta$ such that $h_{1,2}=\alpha h_{3,4}\text{ and } h_{1,3}=\beta h_{2,4}.$
Then $\left(\dfrac{h_1}{h_4}\right)^{2}$ is a small function, and so is $\dfrac{h_1}{h_4}$. Hence $f$ is a quasi-M\"{o}bius transformation of $g$. This is a contradiction.

From the above two cases, we always get a contradiction. Hence $f$ is a quasi-M\"{o}bius transformation of $g$. 
\end{proof}

\section*{Data availability}
Data sharing is not applicable to this article as no datasets were generated or analyzed during the current study.

\section*{Disclosure statement}
No potential conflict of interest was reported by the author(s).

\section*{Funding Acknowledgements}
No funding was received for this research.


\begin{thebibliography}{99}

\bibitem{QA} N. V. An and S. D. Quang, \textit{Two meromorphic functions sharing four pairs of small functions}, Bull. Korean Math. Soc. \textbf{54} (2017) 1159--1171.

\bibitem{CC} H. Z. Cao and T. B. Cao, \textit{Two meromorphic functions share some pairs of small functions with truncated multiplicities}, Acta Math. Sci. \textbf{34} (2014) 1854-1864.

\bibitem{CG} T. Czubiak and G. Gundersen, \textit{Meromorphic functions that share pairs of values}, Complex Var. Elliptic Equ. \textbf{34} (1997), 35-46.
\bibitem{LY95} P. Li and C. C. Yang, \textit{Some further results on the unique range set of meromorphic functions}, Kodai Math. J., \textbf{18} (1995), 437--450.

\bibitem{LY1} P. Li and C. C. Yang, \textit{On two meromorphic functions that share pairs of small functions}, Complex Var. Elliptic Equ. \textbf{32} (1997), 177-190.

\bibitem{LY2} P. Li and C. C. Yang, \textit{Meromorphic functions that share some pair of small functions}, Kodai Math. J. \textbf{32} (2009), 130--145.

\bibitem{LZ} P. Li and Y. Zhang, \textit{Meromorphic functions whose derivatives share small functions}, Kodai Math. J. \textbf{27} (2004), 261--271.

\bibitem{N} J. Noguchi, \textit{A note on entire pseudo holomorphic curves and the proof of Cartan-Nochka's theorem}, Kodai Math. J. \textbf{28} (2005), 336-346.

\bibitem{NO} J. Noguchi and T. Ochiai, \textit{Introduction to Geometric Function Theory in Several Complex Variables}, Trans. Math. Monogr. 80, Amer. Math. Soc., Providence, Rhode Island, 1990.

\bibitem{Q12} S. D. Quang, \textit{Two meromorphic functions share some pairs of small functions,} Compl. Anal. Oper. Th. \textbf{7} (2013), 1357-1370.

\bibitem{QQ} S. D. Quang and L. N. Quynh, \textit{Two meromorphic functions sharing some pairs of small functions regardless of multiplicities}, Internat. J. Math. \textbf{25} (2014) 1450014 (16 pages).

\bibitem{QK} S. D. Quang and N. X. Ky, \textit{On the characteristic functions of meromorphic functions sub-weighted sharing three values}, J. Math. Anal. Appl. Volume \textbf{538}, Issue 1, (2024) 128356.

%\bibitem{QP} S. D. Quang and P. Nguyen, \textit{Meromorphic functions weakly sharing four pairs of small functions with bi-weights}, Preprint (2024).

\bibitem{Y} K. Yamanoi, \textit{The second main theorem for small functions and related problems}, Acta Math. \textbf{192} (2004), 225-294.

\bibitem{ZY} J. Zhang and L. Yang, \textit{Meromorphic functions sharing pairs of small functions}, Math. Slovaca \textbf{65} (2015), 93-102.
\end{thebibliography}
\end{document}